\documentclass{article}
\usepackage{authblk}

\usepackage[english]{babel}
\usepackage{amsmath,amssymb,stmaryrd,mathrsfs,xcolor,latexsym,theorem}
\usepackage{geometry}
\usepackage{todonotes}

\usepackage{ulem}
\usepackage{xcolor} 
\usepackage{hyperref}

\newtheorem{assumption}{Assumption}

\newcommand{\ome}{\omega}
\newcommand{\E}{\mathbb{E}}
\newcommand{\PP}{\mathbb{P}}
\newcommand{\vep}{\varepsilon}

\newcommand{\BY}{\mathbf{Y}}

\newcommand{\assign}{:=}
\newcommand{\backassign}{=:}
\newcommand{\cdummy}{\cdot}
\newcommand{\nobracket}{}
\newcommand{\op}[1]{#1}

\newcommand{\tmmathbf}[1]{\ensuremath{\boldsymbol{#1}}}
\newcommand{\tmop}[1]{\ensuremath{\operatorname{#1}}}
\newcommand{\tmtextbf}[1]{\text{{\bfseries{#1}}}}
\newcommand{\tmtextit}[1]{\text{{\itshape{#1}}}}
\newenvironment{itemizeminus}{\begin{itemize} }{\end{itemize}}
\newenvironment{proof}{\noindent\textbf{Proof\ }}{\hspace*{\fill}$\Box$\medskip}
\newtheorem{thm}{Theorem}[section]
\newtheorem{theorem}[thm]{Theorem}

\newtheorem{lemma}[thm]{Lemma}

\newtheorem{corollary}[thm]{Corollary}

\newtheorem{remark}[thm]{Remark}

\newtheorem{proposition}[thm]{Proposition}

\newtheorem{definition}[thm]{Definition}

\numberwithin{equation}{section}

\usepackage{amsmath,tikz}
\newcommand{\fatnorm}[1]{%
  \begin{tikzpicture}[baseline={([yshift=-.5ex]current bounding box.center)}]
    \node[inner sep=0pt, outer sep=0pt] (m) {$#1$};
    \draw[line width=0.8mm] ([xshift=-1.5mm]m.south west) -- ([xshift=-1.5mm]m.north west);
    \draw[line width=0.8mm] ([xshift=1.5mm]m.south east) -- ([xshift=1.5mm]m.north east);
  \end{tikzpicture}
}

\begin{document}

\title{Randomisation of rough stochastic differential equations}


\author{Peter K.~Friz}
\affil{TU Berlin and WIAS Berlin}

\author{Khoa L\^e}
\affil{School of Mathematics, University of Leeds}

\author{Huilin Zhang}
\affil{Shandong U. and Humboldt U.}


\maketitle

\begin{abstract}
  Rough stochastic differential equations (RSDEs) are  common generalisations
  of It{\^o} SDEs and Lyons RDEs and have emerged as new tool in  several areas
  of applied probability, including non-linear stochastic filtering, pathwise
  stochastic optimal control, volatility modelling in finance and mean-fields
  analysis of common noise system.
  
  We here take a unified perspective on rough It{\^o} processes  and discuss in
  particular when and how they become, upon randomisation, ``doubly
  stochastic'' It{\^o} processes, and what can be said about their conditional laws. 
\end{abstract}


\tableofcontents

\section{Introduction}

Inspired by the recent theory of rough stochastic differential equations (RSDEs) of \cite{FHL21x},
\textit{rough It{\^o} processes} are continuous adapted stochastic processes, on some stochastic basis $(\Omega, \mathfrak{F},
(\mathfrak{F}_t)_{t \geq 0}, \mathbb{P})$, of the form $X = X (\omega, \mathbf{Y}) = X^{\mathbf{Y}} (\omega)$ with dynamics\footnote{All processes are multidimensional; extension to infinite dimensions are possible but not the purpose of this note.}
\begin{equation}
  \label{RIP1} d X_t = A_t (\omega, \mathbf{Y}) d t + \Sigma_t (\omega,
  \mathbf{Y}) d B_t + (F_t, F'_t) (\omega, \mathbf{Y}) d \mathbf{Y}_t,
  \quad 0 \leqslant t \leqslant T,
\end{equation}
where $B = B (\omega)$ is Brownian motion and $\mathbf{Y} = (Y, \mathbb{Y})
\in \mathscr{C}   {  \equiv \mathscr{C}_T  \subset  \mathscr{C}^{\alpha}  ([0, T], \mathbb{R}^{d_Y})}$, a class of deterministic rough paths, on a fixed H\"older 
scale with exponent $\alpha \in (1/3,1/2]$. All coefficients are assumed progressive
in $(t, \omega)$ and $(F_t, F'_t) (\omega, \mathbf{Y})$ is a stochastic
controlled rough path in the sense of \cite{FHL21x}; this allows to give
bona fide integral meaning to \eqref{RIP1} where the reader can also find an It{\^o} formula for such processes. 
\eqref{RIP1} generalizes classical
It{\^o}-diffusions $X = X (\omega)$, with dynamics
\[ d X_t = a_t (\omega) d t + \sigma_t (\omega) d B_t ,\]
which of course accommodates solutions to It{\^o} stochastic differential equations,
\[ d X_t = b_t (X_t, \omega) d t +  \sigma_t (X_t, \omega) d B_t . \]
From a rough path perspective, (\ref{RIP1}) generalizes rough integration, with dynamics
\begin{equation}
  \label{Dav} d X_t = (F_t, F'_t) d \mathbf{Y}_t \Leftrightarrow (\delta
  X)_{s, t} \approx_{3 \alpha} F_s (\delta Y)_{s, t} + F'_s \mathbb{Y}_{s, t},
  (\delta F)_{(s, t)} \approx_{2 \alpha} F_s' (\delta Y)_{s, t},
\end{equation}
where $(F, F') = (F, F') (\mathbf{Y})$ is a given $Y$-controlled rough path.\footnote{ By $ \Psi \approx_{\gamma } \Phi$ we mean $| \Psi_{s,t} - \Phi_{s,t} | \le C |t-s|^\gamma$, for some constant $C$, uniformly on $[0,T]$.}
This accommodates solutions to ``contextualized'' (by which we mean explicit
$\mathbf{Y}$-dependence in the coefficient fields) rough differential equation (RDE)
\begin{equation}
  \label{RDE1} d X_t = f (X_t ; \mathbf{Y}) d \mathbf{Y}_t
  \end{equation}
via the usual Davie expansion $F_t \assign f (X_t ; \mathbf{Y}), \  F'_t
\assign ((D f) f) (X_t ; \mathbf{Y})$. This also accommodates the case of regular time-dependence $f (t,X_t ; \mathbf{Y})$  in \eqref{RDE1}; and more generally RDEs of the form
\begin{equation}
  \label{RDE2} d X_t = (f_t, f_t') (X_t ; \mathbf{Y}) d {\mathbf{Y}_t}  ,
\end{equation}
  in which case $F_t =f_t (X_t ; \mathbf{Y}), \ F'_t = ((D f_t) f_t + f_t')
(X_t ; \mathbf{Y})$. 
 Our main motivation for $\eqref{RIP1}$ comes from the
recent theory of rough SDEs \cite{FHL21x}
\begin{equation}
  \label{RSDE1} d X_t = b_t (X_t, \omega ; \mathbf{Y }) d t + \sigma_t (X_t,
  \omega ; \mathbf{Y }) d B_t + (f_t, f'_t) (X_t, \omega ; \mathbf{Y}) d
  \mathbf{Y}_t,
\end{equation}
solutions $X = X^{\mathbf{Y}}$ of which provide natural examples of rough
It{\^o} processes, with
\[ (F_t, F'_t) (\omega, \mathbf{Y}) = (f_t (X_t, \omega ; \mathbf{Y}),\ ((D
   f_t) f_t + f_t') (X_t, \omega ; \mathbf{Y})) . \]
Extensions to mean-field a.k.a. McKean-Vlasov rough SDEs have been developed in \cite{FHL25p}. 

The purpose of this note is to systematically relate the ``hybrid'' rough stochastic processes seen above to their ``doubly'' stochastic counterparts, when $\mathbf{Y}$ above is replaced
by some stochastic driving noise, often (but not always) an independent Brownian noise (seen as rough path). Conversely, one can start with doubly stochastic situations and ask what happens under partial conditioning 
of the noise. To the best of our knowledge, such questions were first raised in the context of filtering \cite{CDFO13}. 

In Section \ref{sec:mot}, we give motivation from non-linear stochastic filtering, volatility modelling in finance, pathwise stochastic optimal control and mean-fields SDE with common noise system; topics also closely related to \cite{CD18,CD18b}. We do not strive for precise statements, but try to explain how rough stochastic processes naturally emerge, together with motivation for what is to come: Section \ref{sec:RIP} deals with the mathematics of randomisation of rough It{\^o} processes, one obstacle being the potentially uncountable number of rough paths that figure as parameter in \eqref{RIP1}. 
The remaining section discuss applications to RSDE and then also mean-field RSDEs. 
The body of work surrounding RSDEs, control and common noise also resonates with early remarks made by the authors of \cite{CD14}, conjecturing the use of rough paths for conditioned forward stochastic dynamics (in their context of mean field games with common noise). 

\medskip
\noindent {\bf Acknowledgement}: PKF and HZ acknowledge support from DFG CRC/TRR 388 ``Rough
Analysis, Stochastic Dynamics and Related Fields'', Projects A07, B04 and B05. Part of this work was carried out during a visit of the first author to Shandong University. 
KL acknowledges supports from EPSRC
[grant number EP/Y016955/1]. HZ is partially supported by the Fundamental Research Funds for the Central Universities, NSF of China and Shandong (Grant Numbers 12031009, ZR2023MA026), Young Research Project of Tai-Shan (No.tsqn202306054).

\section{Motivating examples} \label{sec:mot}

\subsection{Non-linear stochastic filtering}

We follow \cite{CDFO13, BFLZ25p}.
One has an observation process $Y_t = \int_0^t h (X_s, Y_s) ds
+ B^{\perp}_t$ where the signal process $X$ is assumed to have ``correlated''
dynamics,
\begin{align} \label{equ:sig}
d X_t & = \tilde{b} (t, X_t, Y_t) d t + \sigma (t, X_t, Y_t) d B_t
        (\omega) + f (t, X_t, Y_t) d B^{\perp}_t (\omega)\nonumber \\
      & = b (t, X_t, Y_t) d t + \sigma (t, X_t, Y_t) d B_t (\omega) + f (t,
        X_t, Y_t) d Y_t (\omega) .                                
        \end{align}                                               
        Here, $(B, B^{\perp})$ and $(B, Y) $ are independent Brownians under some
        original measure $\mathbb{P}^o$ and $\mathbb{P}$, respectively,
        related by the Girsanov formula                           
        \begin{eqnarray*}                                         
        \left. \frac{d\mathbb{P}^o}{d\mathbb{P}} \right|_{\mathcal{F}_t} = \exp
        \left( \int_0^t h (s, X_s, Y_s) d Y_s - \frac{1}{2} \int_0^t | h (s, X_s,
        Y_s) |^2 ds \right) \backassign \exp (I_t) .              
        \end{eqnarray*}                                           
        Formally at least, the conditional dynamics of $X_t$ given the observation $\{
        Y_t : 0 \leqslant t \leqslant T \}$ should be captured by the solution $X^{\mathbf{Y}}$ to a rough SDE of the form \eqref{RSDE1}\footnote{Assuming regular $t$ dependence in $f$.}
        \begin{equation}                                          
        \label{RSDE2}                                             
        d X^{\mathbf{Y}}_t = b (t, X^{\mathbf{Y}}_t, Y_t) d t + \sigma (t, X^{\mathbf{Y}}_t, Y_t) d B_t (\omega) +
        (f, D_y f) (t, X^{\mathbf{Y}}_t, Y_t) d \mathbf{Y}_t,     
        \end{equation}                                            
        together with the ``rough stochastic'' Girsanov exponent  
        \[ I_t^{\mathbf{Y}} = \int_0^t (h,(D_x h f + D_y h)) (s, X^{\mathbf{Y}}_s, Y_s)
        d                                                         
        \mathbf{Y}_s - \frac{1}{2} \int_0^t | h (s, X^{\mathbf{Y}}_s, Y_s) |^2
        ds. \]                                                    
        Bearing exponential integrability, one defines a flow of measures
        \[ \langle \mu^{\mathbf{Y}}_t, \varphi \rangle            
        \assign                                                   
        \mathbb{E} [\varphi (X^{\mathbf{Y}}_t)                    
        \exp (I^{\mathbf{Y}}_t) ],                                
        \]                                                        
        whose randomisation $\bar{\mu}_t (\omega) \assign         
        \mu^{\mathbf{Y}}_t | \nobracket_{\mathbf{Y} = \mathbf{Y}^{\tmop{Ito}}
        (\omega)}$ should satisfy                                 
        \[ \langle \bar{\mu}_t (\omega), \varphi \rangle \assign  
        \mathbb{E} [\varphi (X _t) \exp (I _t) |                  
        \mathcal{F}^Y_t \nobracket] . \]                          
        The interest in this construction is a rough path factorisation of the
        Kallianpur-Striebel formula,                              
        \[ {\mathbb{E}^o} [\varphi (X _t) |                       
        \mathcal{F}^Y_t ] = \left. \frac{\langle \mu^{\mathbf{Y}}_t,
        \varphi \rangle}{\langle \mu^{\mathbf{Y}}_t, 1 \rangle}   
        \right|_{\mathbf{Y} = \mathbf{Y}^{\tmop{Ito}} (\omega)} . \]
        A suitable class of rough paths, for $0 \leqslant t \leqslant T$, is the
        (Polish) rough path space $\mathscr{C}^{0, 2 \alpha} ([0, T],
        \mathbb{R}^{d_Y})$. In fact, one can also work with (cf. Section \ref{sec:not} for notation) \ \
        \[ \mathscr{C}_T = \{ \mathbf{Y} \in \mathscr{C}^{0, 2 \alpha} ([0, T]) :
        [\mathbf{Y} ]_{s, t} = I_{d_Y\times d_Y} \times (t - s), 0 \leqslant s \leqslant t
        \leqslant T\} \]                                          
        to hard-wire the bracket of the observation (rough) path to be consistent with
        the quadratic variation of the observation process: $\langle Y \rangle_{s, t}
        = \langle B^{\perp} \rangle_{s, t} = I_{d_Y\times d_Y} \times (t - s)$. Equivalently, one can
        recast $(X, I)$ above in Stratonovich form w.r.t. $Y$; as was done in \cite{CDFO13} to study Clark's robustness problem.
        The appropriate                                           
        randomisation is then $\mathbf{Y} = \mathbf{Y}^{\tmop{Strato}} (\omega)$,
        so that one can take $\mathscr{C}_T = \mathscr{C}^{0, 2 \alpha}_g ([0, T],
        \mathbb{R}^{d_Y})$, the space of geometric rough paths. We mention that under natural assumptions, $(\mu^{\mathbf{Y}}_t)$ is the unique solution to an intrinsically defined rough Zakai equation \cite{BFS24x, BFLZ25p}.
        (Specialists may appreciate the fact that uniqueness results can be obtained by forward-backward arguments, without forward-backward filtration issues
        that arise in presence of $Y$-dependent coefficients in \eqref{equ:sig}; cf. \cite{CP24x} for a recent discussion.)

\subsection{Pathwise stochastic optimal control}

We follow many authors (Davis-Burstein,  Lions-Souganidis, Buckdahn-Ma, Rogers, ... ), leaving detailed references to \cite{FLZ24x},  in
investigating the random field
\begin{equation}
  \label{sVF} V (t, x ; \omega) : = \tmop{essinf}_{\eta (.)}
  \mathbb{E}^{t, x} \left( g \left( {X_T^{\eta}} 
  \right)  | \mathcal{F}_T^W  \right),
\end{equation}
where $X^{\eta}$ has controlled stochastic dynamics\footnote{As noted in \cite{DFG17}, one must not control the final coefficient field $f$
for otherwise the problem degenerates.}
\begin{equation}
  \label{cSDE1} dX^{\eta}_t = b (t, X^{\eta}_t ; \eta_t) dt + \sigma (t, X^{\eta}_t ; \eta_t) d B_t
  + f (t, X^{\eta}_t) d W_t .
\end{equation}
In some formal sense at least, the conditional dynamics of $X^{\eta}$, given some fixed
environmental noise $\{ W_t : 0 \leqslant t \leqslant T \}$, should be captured
by the solution $X^{\eta, \mathbf{Y}}$ to the controlled rough SDE\footnote{As explained in \eqref{RDE1},\eqref{RDE2}, having $f' \equiv 0$, amounts to a regular $t$ dependence in $f$ which is assumed here.}
\begin{equation}
  \label{RSDE3} dX^{\eta, \mathbf{Y}}_t = b (t, X^{\eta, \mathbf{Y}}_t ; \eta_t) dt + \sigma (t, X^{\eta, \mathbf{Y}}_t ; \eta_t) d B_t
  + (f, 0) (t, X^{\eta, \mathbf{Y}}_t) d \mathbf{Y}_t .
\end{equation}
With $\eta = \eta_t (\omega ; \mathbf{Y})$ these are RSDEs of the form \eqref{RSDE1}.
One reverts to a deterministic value function
\begin{equation}
  \label{rVF} V (s, x ; \mathbf{Y}) : = \inf_{\eta (.)} \mathbb{E}^{s, x}
  \left( g \left( {X_T^{\eta, \mathbf{Y}}}  \right)
  \right),
\end{equation}
with immediate regularity results (derived from estimates in \cite{FHL21x} for solutions of RSDEs), validity of a
dynamic programming principle, existence of $\varepsilon$-optimal controls $\eta = \eta_t (\omega ;
\mathbf{Y})$ etc. Rough HJB equations for \eqref{rVF} have been considered several authors; again referring to \cite{FLZ24x} for up-to-date references. (See also \cite{DFG17, HZ25p} for
a Pontryagin Maximum Principle view, relying on linear RSDE theory \cite{BCN24x}.)
 Generalising previous insights of \cite{BM07}, we saw that the randomised value function 
\[ \bar{V} (s, x, \omega) \assign V (s, x ; \mathbf{Y}) | 
   \nobracket_{\mathbf{Y} = \mathbf{W}^{\tmop{Ito}} (\omega)} \]
coincides with \eqref{sVF} as long as the controls in (\ref{cSDE1}) do not
anticipate the noise factor $B.$

The use of the RSDEs has a liberating effect on the modelling of the
conditional stochastic dynamics of (\ref{cSDE1}): applied with $\mathbf{Y} =
\{ \mathbf{W}_t^{\tmop{Ito}} (\omega) : 0 \leqslant t \leqslant T \}$ and
$\eta = \eta_t (\omega ; \mathbf{Y})$ the randomisation
\[ \bar{X}_t (\omega) \assign \left( {X_t^{\eta  (\omega ;
   \mathbf{Y}) ; \mathbf{Y}}}  (\omega) \right) |  \nobracket_{\mathbf{Y} =
   \mathbf{W}^{\tmop{Ito}} (\omega)} \]
is immediately (if subject to some measurability considerations) defined,
without resorting to any additional theory of ``anticipating'' generalisations
of (\ref{cSDE1}).

In fact, all results described remain valid if the Brownian statistics of $W$
in (\ref{cSDE1}) are replaced by a general rough path noise $\mathbf{W}
(\omega)$, confirming one's intuition that the Brownian nature of $W$ in
(\ref{cSDE1}) cannot be of much importance when it is looked upon
conditionally.

\subsection{Volatility modelling in finance} \label{sec:vmf}

A generic (possibly non-Markovian) \tmtextit{stochastic volatility }model,
built on two independent Brownians $(B, W)$, is of the form
\[ X_t = \int_0^t \sqrt{V_s (\omega)} \left( \sqrt{1 - \rho_s^2} d B_s +
   \rho_s d W_s \right) \]
with deterministic (often constant) correlation $\rho_t$ and a non-negative continuous
$(\mathcal{F}_t^W)$-adapted process $V_t (\omega)$. This generalizes the
Bachelier model (when $V_t$ is constant) and accommodates many classical
Markovian as well as recent non-Markovian stochastic volatility models, including (``Bachelier
variants'' of) Heston, Bergomi, Stein-Stein, rough Heston, rough Bergomi, etc.
An important generalisation are \tmtextit{local stochastic volatility} models. Leaving detailed references to \cite{BBFP25}, these are of form\footnote{From a financial modelling perspective, the right-hand side of \eqref{equ:LSV} covers ``local'' correlation, i.e. $\rho = \rho(t,X_t).$} 
\begin{equation} \label{equ:LSV}
d X_t = \ell (t, X_t) \sqrt{V_t (\omega)} \left( \sqrt{1 - \rho_t^2} d B_t
   + \rho_t d W_t \right) \backassign \sigma_t (X_t ; V_t (\omega)
   ) d B_t + f (t, X_t) d M_t \end{equation}
where we introduced the local martingale $M_t \assign \int_0^t \sqrt{V_s
(\omega)} d W_s$. For continuous $V$, we see that $[M] \in C^1$, with $V_t
(\omega) = \partial_t [M]_t .$ If one is interested in the conditional law of
$X_t, 0 \leqslant t \leqslant T_,$ given the information generated by $W$ up
to time $T$, one is naturally led to RSDEs of the form \eqref{RSDE1}, namely\footnote{We write again $(f,0)$ rather than $f$ to indicate regular $t$ dependence in $f$.}
\begin{equation}
  \label{rLSV} d X^{\mathbf{Y}}_t = \sigma_t (X^{\mathbf{Y}}_t ; \partial_t
  [\mathbf{Y}]_t) d B_t + (f,0) (t, X^{\mathbf{Y}}_t) d \mathbf{Y}_t,
\end{equation}
for a suitable class of rough paths like
\[ \mathscr{C}_T = \{ \mathbf{Y} \in \mathscr{C}^{0, 2 \alpha} ([0, T]) : t
   \mapsto [\mathbf{Y} ]_{0, t} \in C^1 ([0, T]) \}  . \]
Formally at least,
\[ \mathbb{E} (g (X_T^{t, x}) | \mathcal{F}_T^W \nobracket) =\mathbb{E}^{t,
   x} \left( g \left( {X^{\mathbf{Y}}_T}  \right) \right) |_{\mathbf{Y} =
   \tmmathbf{M}^{\tmop{Ito}}  (\omega)} \nobracket \]
which generalizes conditional option pricing formulas as seen in classical finance literature \cite{HW87, RT97}.

We remark that even if  the original dynamics \eqref{equ:LSV} are highly non-Markovian\footnote{E.g. in the popular situation when $V$ is modelling by a fractional Brownian motion with small Hurst parameter},
the roughpath-ified 
equation \eqref{rLSV} describes a Markov process, which leads then to a (rough) Feynman-Kac formalism that connects expression like
$\mathbb{E}^{t, x} g \left( {X^{\mathbf{Y}}_T}  \right) $ to
rough partial differential equations of Kolmogorov backward type
{\cite{BBFP25}} and {\cite{BFS24x}}.

\subsection{Mean-fields analysis of common noise system}
Let $B,W$ be two independent Brownian motion and consider\footnote{For simplicity of notation, assume autonomous coefficients $b, \sigma, f$, i.e. without explicit $t$-dependence. That said, the
dependence on $\mu_t (\omega)$, or $\mu^{\mathbf{Y}}_t$ below, automatically leads to a (non-regular) $t$-dependence, discussed below.} 
\begin{equation}
  \label{comMKV} d X_t = b (X _t, \mu_t (\omega)) {d t + \sigma
  }  (X _t, \mu_t (\omega)) d B_t + f  (X _t, \mu_t (\omega)) d W_t, \quad
  \mu_t (\omega) \assign \tmop{Law} (X_t | \mathcal{F}^W_t \nobracket).
\end{equation}
Such system arise naturally as the mean-field $(N \rightarrow \infty)$ limit
of interacting particle systems where each particle (or agent) has dynamics 
\begin{align}\label{eqn.Npart}
  {d X_t^{N, i}} = b \left( {{X_t^{N, i}}} ,
   \mu^N_t (\omega) \right) {d t + \sigma}  \left( {{X_t^{N,
   i}}} , \mu^N_t (\omega) \right) d B^i_t + f  \left( {{X_t^{N,
   i}}} , \mu^N_t (\omega) \right) d W_t \quad
\end{align}
with independent ${X_0^{N, i}} \sim X_0$, idiosyncratic Brownian
noises  ($B^i$)  and common noise $W$, all assumed independent. Here, $\mu^N_t
(\omega)$ is the empirical measure of $\left\{ {X_t^{N, i}} : 1
\leqslant i \leqslant N \right\}$.

Formally at least, the conditional dynamics of $X_t$ given some fixed
environmental noise $\{ W_t : 0 \leqslant t \leqslant T \}$ should be captured
by the solution \ $X^{\mathbf{Y}}$ to the mean-field RSDE
\begin{equation}
  \label{mfRSDE2} d X^{\mathbf{Y}}_t = b (X^{\mathbf{Y}}_t, \mu^{\mathbf{Y}}_t) {d t
  + \sigma}  (X^{\mathbf{Y}}_t, \mu^{\mathbf{Y}}_t) d B_t + f
  (X^{\mathbf{Y}}_t, \mu^{\mathbf{Y}}_t) d \mathbf{Y}_t, \quad
  \mu_t^{\mathbf{Y}} \assign \tmop{Law} (X^{\mathbf{Y}}_t).
  \end{equation}
  We then expect to see that
\[ \tmop{Law} (X^{\mathbf{Y} }_t) |_{\mathbf{Y} =
   \mathbf{W}^{\tmop{Ito}}} \nobracket = \tmop{Law} (X_t | W \nobracket) .
\]
Following \cite{BCD20, FHL25p} we can in this context utilize the Lions lifts of coefficient fields $\Theta(x,\mu)$, $\Theta \in \{b,\sigma,f\}$, which replace $\mu$ by a r.v. $\xi$ on some auxiliary moment space
$L_q (\Omega')$,
\[ \hat{\Theta} (x, \xi) \assign \Theta (x, \tmop{Law} (\xi)) . \]
Any solution to 
\eqref{mfRSDE2} is then understood as a rough It{\^o} process with, setting
$\mathcal{X}_t \assign (X^{\mathbf{Y}}_t (\omega), X^{\mathbf{Y}}_t (\cdot))$,
\begin{eqnarray*}
  F_t (\omega) & \assign & f (X^{\mathbf{Y}}_t (\omega), \mu_t) = \hat{f}_t
  (\mathcal{X}_t)\\
  F_t' (\omega) & : = & ((D_x  \hat{f}_t) \hat{f}_t) (\mathcal{X}_t) + (D_{\xi} 
  \hat{f}_t) (\mathcal{X}_t) [\hat{f}_t (\mathcal{X}_t)] 
   \equiv  ((D  \hat{f}_t) \hat{f}_t) (\mathcal{X}_t). 
\end{eqnarray*}
In the above, $D=(D_x,D_{\xi})$,  $D_x$ is the usual Fr\'echet derivative while $D_\xi$ is the Fr\'echet derivative along the subspace $L_p(\Omega')$ (with some $p\ge1$) of $L_q(\Omega')$  which contains the direction of $\xi$-derivative $\hat{f}(\mathcal{X}_t)$. 
In classical treatments, one often takes $p=q=2$, which rules out most coefficient fields due to their lack of higher differentiability on $L_2$, with examples given in \cite{CD18}. 
This issue is resolved in \cite{FHL25p} by choosing $p$ sufficiently large, hence, effectively enlarging the class of coefficient fields which are three-times Fr\'echet differentiable along $L_p$-directions  (recall that three orders of derivatives are sufficient to solve \eqref{mfRSDE2} as a rough SDE uniquely).
Remark that the results of \cite{CL14, BCD20} do not cover ours and conversely. (Specifically, \cite{BCD20} consider mean field RDEs driven by classes of random rough paths; their processes are not rough It{\^o} processes and there setup not designed to deal with common noise and conditional laws.) From \cite{CG19, CN21}, we can then expect that 
$\mu_t^{\mathbf{Y}} = \tmop{Law} (X^{\mathbf{Y}}_t)$ solves a non-linear, non-local rough partial
differential equations of Kolmogorov forward type, uniquely so under natural assumptions, as is shown in \cite{BFS25p}.



\section{Rough It{\^o} Processes} \label{sec:RIP}

\subsection{Notation and background} \label{sec:not}

{\bf Stochastic processes.} 
A filtered probability space is denoted by $(\Omega, \mathfrak{F},
(\mathfrak{F}_t)_{t \geq 0}, \mathbb{P})$; the filtration $(\mathfrak{F}_t) \subset \mathfrak{F}$
is \tmtextit{complete} if $\mathfrak{F}_0$ (and thus all ${\mathfrak{F}_t} $)
contains all $\mathbb{P}$-negligible sets and \tmtextit{right-continuous} if
$\mathfrak{F}_t = \bigcap_{s > t} \mathfrak{F}_s$ for all $t \geq 0$. 

A topological space $U$ is equipped with its \(\sigma\)-algebra of  Borel sets,
$\mathfrak{U}=\mathfrak{B} (U)$; if  $U$ Polish (complete, separable, metric space)
call $(U, \mathfrak{U}) = (U, \mathfrak{B}(U))$ a
\tmtextit{Polish measurable space}. Given generic measurable spaces $(M_i,
\mathfrak{M}_i), i = 1, 2$, we say $f : M_1 \to M_2$ is $\mathfrak{M}_1
/\mathfrak{M}_2$-measurable if $f^{- 1} (\mathfrak{M}_2) \subset
\mathfrak{M}_1$. A process $X = (X_t)_{t \geq 0}$ with values in a measurable
space $(M, \mathfrak{M})$ is said to be \tmtextit{measurable} if the mapping
$(\omega, t) \mapsto X_t (\omega)$ defined on $\Omega \times \mathbb{R}_+$
equipped with the product $\sigma$-field $\mathfrak{F} \otimes \mathfrak{B}
(\mathbb{R}_+)$ is measurable, \tmtextit{adapted} if $X_t$ is
$\mathfrak{F}_t$-measurable, for all $t \geq 0$, and \tmtextit{progressively
measurable} (or  \tmtextit{progressive}) if for each $t$, the mapping $(s, \omega) \mapsto X_s (\omega)$ on
$[0, t] \times \Omega$ is $(\mathfrak{B}_t \otimes
\mathfrak{F}_t)$-measurable, were $\mathfrak{B}_t=\mathfrak{B}([0, t])$. For (say, $\mathbb{R}^d$-dimensional) processes,  the 
 \tmtextit{optional} $\sigma$-field
$\mathfrak{O}$ is the smallest $\sigma$-field making all adapted,
c{\`a}dl{\`a}g processes measurable, the \tmtextit{progressive} $\sigma$-field
$\Pi$ is the smallest $\sigma$-field making all progressively
measurable processes measurable. 
Recall $\mathfrak{O}  \subset \Pi$. Also, we will write  $\mathfrak{O}_T,   \Pi_T$ if only interested on a finite horizon $[0,T]$. 
We will be interested in processes depending on a parameter. Given a measurable space $(U, \mathfrak{U})$, we say that a process $X$ on $U
\times \mathbb{R}_+$ is $\mathfrak{U}$-progressive if it is $(\mathfrak{U} \otimes \Pi)$-measurable as function of $(u,t,\omega)$. Replacing $\Pi$
by $\mathfrak{O}$ leads to the notation of  $\mathfrak{U}$-optional processes. (Similarly, one can define $\mathfrak{U}$-predictable processes, of use in presence of jumps \cite{AP24x}.)
\medskip 

\noindent {\bf Measurable selection.}  It is known (see e.g. \cite[Thm 18.25]{Kal21}) that the stochastic integral of a $\mathfrak{U}$-progressive process against a continuous semimartingale admits a  $\mathfrak{U}$-progressive version, which is a.s. continuous at each $u$. Similar results for $\mathfrak{U}$-optional (resp. $\mathfrak{U}$-predictable) integrands are due to Stricker-Yor, cf.  \cite[appendix]{FLZ24x} for an (English) summary and precise references.

%
%

\medskip 

\noindent {\bf Rough paths.} \cite{Lyo98,FH20}. We consider a class of level-2 rough path over $\mathbb{R}^{d_Y}$, say
$\mathbf{Y} = (Y, \mathbb{Y}) \in (\mathscr{C}, \rho)$ where $(\mathscr{C}, \rho)$  
is a complete and metric space of rough paths, such that $(\mathscr{C}, \rho) \hookrightarrow (\mathscr{C}^{\alpha} ,
\rho_{\alpha})$, the usual $\alpha$-H\"older rough path space\footnote{ $\rho_{\alpha} (\mathbf{Y}, \overline{\mathbf{Y}}) = |
  \delta Y - \delta \bar{Y} |_{\alpha} + |\mathbb{Y}- \bar{\mathbb{Y}} |_{2
  \alpha}$. Recall also the homogeneous rough path norm 
$ \interleave \mathbf{Y} \interleave_\alpha = | \delta Y|_{\alpha} \vee
   \sqrt{|\mathbb{Y}|_{2 \alpha}}$. }, for some fixed
$\alpha \in (1 / 3, 1 / 2]$. We write $\mathscr{C}_T$ and
$\mathscr{C}^{\alpha}_T$ respectively when we want to  emphasize the time-horizon
$[0, T]$, the notation $\mathscr{C}^{\alpha}  ([0, T], \mathbb{R}^{d_Y})$ will also be used. \ For $\mathbf{Y} \in \mathscr{C}^\alpha$, the bracket
$[\mathbf{Y}]$ is a $2 \alpha$-H\"older path defined by
\[ (\delta Y) \otimes (\delta Y) =\mathbb{Y}+\mathbb{Y}^{\top} + \delta
   [\mathbf{Y}]. \]
We set  $\mathbf{Y}^{\circ} = (Y, \mathbb{Y}^{\circ}) \assign \mathbf{Y} + (0, \delta [\mathbf{Y}] / 2) = 
(Y, \mathbb{Y} + \delta [\mathbf{Y}] / 2)$, noting $[\mathbf{Y^{\circ}}]=0$. Here are some examples of rough path spaces we have in mind, notation along \cite{FH20}.  
\begin{itemize}
  \item $(\mathscr{C}^{\alpha}, \rho_{\alpha})$, the complete (non-separable) space of $\alpha$-H\"older rough paths,
  
  \item $(\mathscr{C}^{\alpha}_g, \rho_{\alpha})$, with $\mathscr{C}^{\alpha}_g \assign \{\mathbf{Y} \in
  \mathscr{C}^{\alpha}: [\mathbf{Y}]\equiv 0 \}$, the complete (non-separable) space of \tmtextit{weakly geometric} $\alpha$-H\"older rough
  paths,
  
  \item $(\mathscr{C}^{0, \alpha}, \rho_{\alpha})$, the Polish space of
  $\alpha$-H\"older rough paths obtained as $\rho_\alpha$-closure of
  smooth rough path (cf. Exercise 2.8 in \cite{FH20}),
  
  \item $(\mathscr{C}^{0,\alpha}_g, \rho_{\alpha})$, the Polish space
  of\tmtextit{ geometric} $\alpha$-H\"older rough paths obtained as $\rho_\alpha$-closure of
  canonically lifted smooth path,

  \item  $(\mathscr{C}^{\alpha, 1}, \rho_{\alpha, 1})$, the complete (non-separable) space of $\alpha$-H\"older rough paths with Lipschitz bracket,
  $\mathscr{C}^{\alpha, 1} \assign \{\mathbf{Y} \in
  \mathscr{C}^{\alpha}: | [\mathbf{Y}] ] |_{\tmop{Lip}} =
 \sup_{s<t} \tfrac{|\delta [ \mathbf{Y}]_{s,t}]|}{|t-s|}  < \infty  \}$ with\footnote{{   Recall $\mathbf{Y} = \mathbf{Y}^\circ - (0, [\mathbf{Y}]/2)$ so 
  that $\rho_\alpha $ and $\rho_{\alpha,1}$ agree on weakly geometric rough paths.}}
$$  \rho_{\alpha, 1} (\mathbf{Y}, \tmmathbf{Z}) \assign \rho_{\alpha}
  (\mathbf{Y}^{\circ}, \tmmathbf{Z}^{\circ}) + | [\mathbf{Y}] -
  [\tmmathbf{Z}] |_{\tmop{Lip}},$$
  
  \item $(\mathscr{C}^{0, \alpha, 1}, \rho_{\alpha, 1})$, the Polish space
  of $\alpha$-H\"older rough paths with continuously differentiable bracket,
  with $\mathscr{C}^{0, \alpha, 1} = \mathscr{} \{\mathbf{Y} \in
  \mathscr{C}^{0, \alpha}: \| [ \mathbf{Y} ] \|_{C^1} < \infty\} $ and metric $\rho_{\alpha, 1}$.
\end{itemize}

\subsection{Measurability of rough It{\^o} processes in (rough path) parameter}

Recall that $\mathscr{C }_T$ denotes a (Polish) space of $(\alpha$-H\"older) rough paths on $[0,
T]$, possibly non-geometric, $\alpha \in (1/3,1/2]$. (The case of one fixed rough path amounts to
taking $\mathscr{C }_T$ as a singleton.) Let $\mathfrak{C}_T$ denote the Borel
sets of $\mathscr{C }_T$. We work on a complete filtered probability space  $\tmmathbf{\Omega} = (\Omega, \mathfrak{F}, (\mathfrak{F}_t)_t,
\mathbb{P})$.

\begin{definition}
  A {\em rough It{\^o} processes} $X$ on $\tmmathbf{\Omega}$ over $\mathscr{C
  }_T$  is a family of continuous adapted processes $X = \{
  X^{\mathbf{Y}}  : \mathbf{Y} \in \mathscr{C }_T \} $, defined on the
  same (filtered) stochastic basis, such that, for each $\mathbf{Y} \in
  \mathscr{C }_T$,
  \begin{equation}
    \label{RIP2} {X^{\mathbf{Y}}_t}  (\omega) {= X^{\mathbf{Y}}_0} 
    (\omega) + \int_0^t A_s (\omega, \mathbf{Y}) d s + \int_0^t \Sigma_s
    (\omega, \mathbf{Y}) d B_s (\omega) + \int_0^t (F_s, F'_s) (\omega,
    \mathbf{Y}) d \mathbf{Y}_s, \quad 0 \leqslant t \leqslant T,
  \end{equation} 
  in a.s. sense, with coefficients fields $\{ A, \Sigma, (F, F') \}$ as
  described below.
\end{definition}
In what follows we write indifferently $\Theta^{\mathbf{Y}}_t (\omega) =
\Theta_t (\omega ; \mathbf{Y})$ for any coefficient field; similar notation
for ${X^{\mathbf{Y}}_0}  (\omega) {= X_0}  (\omega ; \mathbf{Y})$.

\begin{assumption}: \label{ass:1}
For each $\mathbf{Y} \in \mathscr{C }_T$,
\begin{itemizeminus}
  \item all coefficients fields are progressive; i.e. each $(t, \omega) \mapsto
  \Theta^{\mathbf{Y}}_t (\omega)$ is progressively measurable;
  
  \item $A^{\mathbf{Y}}, \left( \Sigma  {\Sigma^T}  \right)^{\mathbf{Y}}
  \in L^1 ([0, T])$ for a.e. $\omega$;
  \item stochastic controlledness of $(F,F')^{\mathbf{Y}}$ in the quantitative form%
  \footnote{{ As in \cite{FHL21x}, we write $\| F \|_{\alpha ; p} = \sup_t \| F_t \|_p + \sup_{s < t} \frac{\| F_t - F_s
\|_p}{| t - s |^{2 \alpha}}$, $\| \mathbb{E}_{\bullet} R^Z \|_{2 \alpha, p} =
\sup_{s < t} \frac{\| \mathbb{E}_s R^Z_{s, t} \|_p}{| t - s |^{2 \alpha}}$. By Thm 3.5 and Cor 3.6 in \cite{FHL21x} condition \eqref{fatnorm} guarantees that $\int (F,F') (Y) d \mathbf{Y}$ is well-defined as limit of partially compensated Riemann sums, in $p$.th mean, by application of (some variant of) the stochastic sewing lemma. We note that \cite{FHL21x} offers more flexibility in choosing H\"older exponents, as well as ``mixed'' integrability exponents, which plays an essential role in the well-posedness theory for RSDEs.} 
  }
    
  
{ 
\begin{equation} \label{fatnorm} 
 \fatnorm{ (F, F')^{\mathbf{Y}} } := \| F \|_{\alpha ; p} + \| F' \|_{\alpha, p} + \| \mathbb{E}_{\bullet} R^F
\|_{2 \alpha ; p} < \infty,
\end{equation}
}
where $F_t (\omega ; \mathbf{Y}) - F_s (\omega ; \mathbf{Y}) = F_s' (\omega
     ; \mathbf{Y}) (Y_t - Y_s) + R_{s, t}^F$. 
     
  \end{itemizeminus}
\end{assumption} 
In many situations one has additional structures which we formalize as follows.

\begin{assumption}: \label{ass:2} Is divided into {\bf (A2.opt)} and {\bf (A2.prog)} as follows. 
All $\Theta$ are $\mathfrak{C}_T$-optional (resp. $\mathfrak{C}_T$-progressive) in the sense of
$(\mathfrak{C}_T \otimes \mathfrak{O}_T)$-measurability (resp. $(\mathfrak{C}_T \otimes \Pi_T)$-measurability) of the map
\[ (\mathbf{Y} ; (t, \omega)) \mapsto \Theta_t (\omega ; \mathbf{Y}). \]
Also, $(\mathbf{Y} ; \omega) \mapsto X_0 (\omega ; \mathbf{Y})$ is
$(\mathfrak{C}_T \otimes \mathfrak{F}_0)$-measurable.\\
\end{assumption} 

\begin{assumption}: \label{ass:3}  For each $\mathbf{Y} \in \mathscr{C}_T$, the coefficient fields are
\tmtextit{causal} in $\mathbf{Y}$. That is, for a.e. $\omega$,
\[ \Theta_t (\omega ; \mathbf{Y}) = \Theta_t (\omega ; \mathbf{Y }^t),
   \qquad \mathbf{Y}^t_s \assign \mathbf{Y}_{s \wedge t} \]
and, similarly, $X_0 (\omega ; \mathbf{Y}) = X_0 (\omega ; \mathbf{Y
}_0)$.
\end{assumption} 

\begin{theorem}
  Assume \eqref{ass:1}. Then the right-hand side of \eqref{RIP2} is well-defined, and
  thus defines a rough It{\^o} process.
\end{theorem}

\begin{proof} %
  Under the conditions put forward in \eqref{ass:1}, the first and second integral are
  well-defined in Lebesgue and It{\^o} sense, respectively, and define 
  continuous adapted processes. The same is true for the final (``stochastic
  rough'') integral, thanks to \eqref{fatnorm} and \cite[Theorem 3.5]{FHL21x}.
\end{proof}

{  \begin{remark} \label{rem:A1forRSDE} (i) Condition \eqref{fatnorm} implies, by Thm 3.5 and  Cor. 3.6 in \cite{FHL21x} the estimate
\begin{equation} \label{equ:RILp} \mathbb{E} \left( \left| \int_s^t (F_r, F'_r) (\omega, \mathbf{Y})
   d\mathbf{Y}_r \right|^p \right) \leqslant C_R \ 
   \fatnorm{ (F, F')^{\mathbf{Y}} }^{\,p}
    \times |t - s|^{\alpha p} 
    \end{equation}
    whenever $\tmmathbf{\interleave Y \interleave}_{\alpha} \leqslant R$, with polynomially growing $C_R$. 
    
    \noindent (ii) Assumption (A1) applies in particular to solutions of rough differential
equations: by the apriori estimates of Proposition 4.5 in \cite{FHL21x}, for $L_{m,\infty}$-solutions, applied with $\beta = \beta' = \alpha$ for simplicity, the integrand of
the rough stochastic integral appearing in the integral formulation of an RSDEs, has finite 
 seminorm
$\|F, F' \|_{Y, 2 \alpha ; m, \infty} = \| \delta F\|_{{\alpha} ; m, \infty} + \|F' \|_{\alpha; m, \infty} +
   \|E_\bullet R^F \|_{2 \alpha ; \infty} $. (We refer to \cite{FHL21x} for details on such mixed integrability space, pivotal for RSDE well-posedness theory.)  Applied with $m=p$, with boundedness of $F$ in the RSDE context, this dominates $\fatnorm{ (F, F') }$ as defined in \eqref{fatnorm}. 
\end{remark} 
}

\begin{theorem}\label{thm:jm}
  (i). (Joint measurability) Assume \eqref{ass:1}+\eqref{ass:2} and write $X = X_t (\omega,
  \mathbf{Y})$ for rough It{\^o} process of the previous theorem. Then there is
  a $\mathfrak{C}_T$-optional (resp. $\mathfrak{C}_T$-progressive) version of $(t, \omega, \mathbf{Y}) \mapsto
  X^{\mathbf{Y}}_t (\omega)$, denoted by the same letter in the sequel.
  
  (ii). (Causality) Assume \eqref{ass:1}+\eqref{ass:3}. Then $X$ is also causal in $\mathbf{Y}$.
   \end{theorem}

\begin{proof}
 (i) By assumption, $A$ and $\Sigma$ are
  $\mathfrak{C}_T$-optional (resp. $\mathfrak{C}_T$-progressive). The desired measurability 
  of the Lebesgue and stochastic integral then follows from results on measurable selection, more
  precisely Lemma 8.5 and Proposition 8.11 of \cite{FLZ24x} under assumption (A2.opt), and \cite[Thm 18.25]{Kal21} in case of (A2.prog).
   
  It remains to understand that the rough stochastic integral as a map
  \[ (t, \omega, \mathbf{Y}) \mapsto \int_0^t (F_s, F'_s) (\omega,
     \mathbf{Y}) d \mathbf{Y}_s \]
admits a $\mathfrak{C}_T$-optional (resp. $\mathfrak{C}_T$-progressive) version. By assumption, $(F, F' )$ is
  $\mathfrak{C}_T$-optional, the existence of a \ $\mathfrak{C}_T$-optionally measurable version of the rough stochastic integral then follows by measurable selection,  
cf. {\cite{FLZ24x}}, Corollary 4.7 for the optional case; similarly, \cite[Proposition 3.11]{HZ25p} for the progressive case. 
  
  (ii) Straight-forward consequence of the definitions of the integrals that
  define $X$.
\end{proof}

\begin{remark} Rough BSDEs \cite{DF12, LT24} provide a natural class of rough It\^o processes for which the causality condition \eqref{ass:3} is not met.\footnote{One can define a notion of backward causality, 
but that would go too far astray here.}
\end{remark}

\subsection{Randomisation and regular conditional distributions}

Let $\tmmathbf{\Omega}' = (\Omega', \mathfrak{G}', (\mathfrak{F}'_t)_t,
\mathbb{P}')$ be a complete probability space, and similar for
$\tmmathbf{\Omega}''$. Write
\[ \tmmathbf{\Omega} = (\Omega, \mathfrak{F}, (\mathfrak{F}_t)_t, \mathbb{P})
\]
for the (completed) product space, where $\Omega = \Omega' \times \Omega''$,
$\mathbb{P}=\mathbb{P}' \otimes \mathbb{P}''$ and $\mathfrak{F}_t
=(\mathfrak{F}'_t \otimes \mathfrak{F}''_t) \vee \mathfrak{N}$ where
$\mathfrak{N}$ denotes the $\mathbb{P}$-negligible sets on $\Omega$. Write
$\omega = (\omega', \omega'')$ accordingly.

Let now $X = \{ X^{\mathbf{Y}}  : \mathbf{Y} \in \mathscr{C}_T \}$ denote
a rough It{\^o} process over $\tmmathbf{\Omega}'$. We further make assumption

\begin{assumption} \label{ass:4}: Assume  $\mathscr{C}_T$ Polish, $(\mathscr{C}_T, \rho) \hookrightarrow (\mathscr{C}_T^{\alpha} ,
\rho_{\alpha})$, $\alpha \in (1/3,1/2)$, cf. Section \ref{sec:not}, and 
\[ \mathbf{W} : (\Omega'', \mathfrak{F}_T'', \mathbb{P}'') \rightarrow
   (\mathscr{C}_T, \mathfrak{C}_T), \]
adapted in the sense that $\mathbf{W}_t \in \mathfrak{F}''_t$. (In other
words, $\mathfrak{F}^{\mathbf{W}}_t \subset \mathfrak{F}_t''$.)\\ 
\end{assumption}

Both can be
``lifted'' to ${\Omega}$, by setting $X^{\mathbf{Y}}  (\omega) \assign
X^{\mathbf{Y}}  (\omega')$ and $\mathbf{W} (\omega) : = \mathbf{W}
(\omega'')$, respectively, where they are, by construction, independent under
$\mathbb{P}$. We also lift $\mathfrak{F}'_t$ (and similarly
$\mathfrak{F}''_t$) in the natural way, in the sense that
\[ A' \in \mathfrak{F}'_t \rightsquigarrow A' \times \Omega'' \in
   \mathfrak{F}'_t \otimes \mathfrak{F}''_0 . \]
We know from Theorem \ref{thm:jm} above that under the appropriate conditions,
a rough It{\^o} process $X = \{ X^{\mathbf{Y}}  : \mathbf{Y} \in \mathscr{C}_T
\}$ over $\tmmathbf{\Omega}'$ admits a $\mathfrak{C}_T$-optional (resp. $\mathfrak{C}_T$-progressive) version, as before coefficient fields $\{ A, \Sigma, (F, F') \}$ denote its coefficient
fields.
%
%
%


\begin{theorem}\label{thm.randomise}
  \label{thm:rcd}Assume \eqref{ass:1}$+$\eqref{ass:2}$+$\eqref{ass:4}.
  
 \noindent (i) Then the randomised rough It{\^o} process 
  \[ \label{randomised-rip} {\bar{X}_t}   (\omega) \assign X_t^{\mathbf{W}
   (\omega'')} (\omega') \assign X_t^{\mathbf{Y}} |_{\mathbf{Y} = \mathbf{W}
   (\omega'')}  \]
   defines a $(\mathfrak{F}'_t \vee \mathfrak{F}''_T : 0
  \leqslant t \leqslant T) $-adapted process.

 \noindent (ii) For any $0 \leqslant t \leqslant T$,
  \[ \label{rcd1} \tmop{Law} (\bar{X}_t | \mathfrak{F}''_T) (\omega) =
     \tmop{Law} (X^{\mathbf{Y}}_t) |_{\mathbf{Y} = \mathbf{W} (\omega )}
     \nobracket, \]
  where $\omega \mapsto \tmop{Law} (X^{\mathbf{Y}}_t) |_{\mathbf{Y} =
  \mathbf{W} (\omega )} \nobracket$ is an explicit regular conditional
  distribution of $\bar{X}_t$ given $\mathfrak{F}''_T$.

 \noindent (iii) Let $\alpha p > 1$. Then the process $\bar{X}$ admits a continuous modification, also denoted by $\bar{X}$.
\end{theorem}

\begin{proof}
  For fixed $t \in [0, T]$, Theorem \ref{thm:jm} implies that $(
  \omega', \mathbf{Y}) \mapsto X^{\mathbf{Y}}_t (\omega')$ is
  $(\mathfrak{F}'_t \otimes \mathfrak{C}_T)$-measurable. Composition with
  $( \omega', \omega'') \mapsto ( \omega', \mathbf{W}
  (\omega''))$, clearly $(\mathfrak{F}'_t \otimes
  \mathfrak{F}''_T)$/$(\mathfrak{F}'_t \otimes \mathfrak{C}_T)$-measurable by
  \eqref{ass:4}, then yields (i).
  
  (ii) 
  Without loss of generality, assume $t = T$. Since $\mathfrak{F}^{\mathbf{W}}_T \subset
  \mathfrak{F}_T''$ we show that for all bounded and continuous functions $g$,
  \[ \label{equiv-condi} \mathbb{E} [g (\bar{X}_T) | \mathfrak{F}''_T]
     =\mathbb{E} [g (X^{\mathbf{Y}}_T)] |_{\mathbf{Y} = \mathbf{W} (\omega)}
     =\mathbb{E}' [g (X^{\mathbf{Y}}_T)] |_{\mathbf{Y} = \mathbf{W}
     (\omega'')} \]
  which amounts to show that for all bounded $\xi \in \mathfrak{F}''_T$
  \[ \mathbb{E} \left[ \xi \hspace{0.17em} \mathbb{E}' [g (X^{\mathbf{Y}}_T)]
     |_{\mathbf{Y} = \mathbf{W} (\omega'')} \right] =\mathbb{E} [\xi g
     (\bar{X} _T)] \]
  which follows from Fubini. Since $\mathfrak{F}^{\mathbf{W}}_T \subset
  \mathfrak{F}_T''$ and $\mathbf{W} \in \mathfrak{F}^{\mathbf{W}}_T$ we also
  have $\tmop{Law} (\bar{X}_t | \mathfrak{F}^{\mathbf{W}}_T) (\omega) =
  \tmop{Law} (X^{\mathbf{Y}}_t) |_{\mathbf{Y} = \mathbf{W} (\omega )}
  \nobracket$. (See \cite[Lem. 8.7]{Kal21} for similar arguments formulated as abstract lemma.)
  To see that this construction yields a regular conditional
  distribution, it remains to see that $\omega \mapsto \tmop{Law}
  (X^{\mathbf{Y}}_t) |_{\mathbf{Y} = \mathbf{W} (\omega )} \nobracket$
  is measurable. But this is obvious from measurability of $\omega \mapsto
  \mathbf{W} (\omega )$, and measurabilty of $\mathbf{Y} \mapsto
  \tmop{Law} (X^{\mathbf{Y}}_t) \in \mathcal{M}_1$, {  where $\mathcal{M}_1$ is the space of probability measures}, equipped with weak
  convergence, the latter being a consequence of Theorem \ref{thm:jm} and
  Fubini.
  
  (iii) Remark that a.s. continuity of $X^{\mathbf{Y}} $, for
  fixed $\mathbf{Y} \in \mathscr{C}_T$, is clear, whereas we cannot hope
  that null sets can be taken independent of $\mathbf{Y}$.
  
  Assume at first that $\tmmathbf{\interleave W \interleave}_{\alpha} \in \cap_{p\ge1}L^p
  $ and that all coefficients fields are bounded in the
  appropriate norms, uniformly in $\mathbf{Y}$. Then, by (ii) for any $0
  \leqslant s \leqslant t \leqslant T$ and any $p < \infty$, and $n \in
  \mathbb{N}$,
  \[ \mathbb{E} \left( \left| \delta {\bar{X}_{s, t}}  \right|^p \wedge n
     \right) =\mathbb{E} \left( \mathbb{E} \left( \left| \delta {\bar{X}_{s,
     t}}  \right|^p \wedge n | \nobracket \mathfrak{F}''_T \right) \right)
     =\mathbb{E} ((\mathbb{E} | \delta X^{\mathbf{Y}}_{s, t} |^p \wedge n)
     |_{\mathbf{Y} = \mathbf{W} (\omega )} \nobracket) . \]
  By monotone convergence, $\mathbb{E} \left( \left| \delta {\bar{X}_{s, t}} 
  \right|^p \right) = \mathbb{E} ((\mathbb{E} | \delta X^{\mathbf{Y}}_{s, t}
  |^p) |_{\mathbf{Y} = \mathbf{W} (\omega )} \nobracket) \lesssim | t - s
  |^{\alpha p}$, where we used basic estimates on Lebesgue integral, the BDG
  inequality, using uniform boundedness of $\Sigma$, and $L^p$-estimate for
  $\int_s^t (F_r, F'_r) (\omega, \mathbf{Y}) d \mathbf{Y}_r,$ as stated in \eqref{equ:RILp}. 
  By Kolmogorov's criterion
  $\bar{X}$ then has a continuous modification.

The general case, with
$\tmmathbf{\interleave W \interleave}_{\alpha}$ is  finite a.s., and the coefficients fields as put forward in \eqref{ass:1}, the proof follows by a localisation inside the Kolmogorov argument. The continuity of Lebesgue integral and stochastic integral parts follows by a standard localisation argument. So without loss of generality, we assume $A=\Sigma=0$. It remains to show the randomised rough integral has a continuous modification. To this end, assume $T=1$, let $D_n$ be the $n$-th dyadic partition of $[0,1]$, and let $D= \cup_{n=1}^{\infty} D_n.$

Step 1. We show $\bar X$ is uniformly continuous on $D$ a.s. Set 
$$
\Omega''_N = \{ \omega'' : \fatnorm{ (F, F')^{\mathbf{Y}} }_{\mathbf{Y} =
\mathbf{W} (\omega'')} + \interleave \mathbf{W} (\omega'') \interleave_{\alpha} \leqslant N
\},
$$ 
and consider $\bar X^N(\ome', \ome''):= \bar X(\ome', \ome'') 1_{\Omega''_N}(\omega'')$. It follows by estimate of stochastic rough integrals again that for any $s,t\in [0,1],$ and since $p>\frac{1}{\alpha}$ 
$$
\E| \delta \bar X^N_{s,t}  |^p=\E''[1_{\Omega''_N} [\E'|\delta X^{\mathbf{Y}}_{s,t}|^p]_{\mathbf{Y}= \mathbf{W}(\ome'') } ] \le C_N |t-s|^{\alpha p}.
$$  
Then it follows by the proof of Kolmogorov's criterion (see e.g. \cite[Theorem 2.9]{LeG16}) that there exists $\Omega^N \subseteq \Omega$ with $\PP(\Omega^N)=1$, such that for any $\ome \in \Omega^N,$ $ \{\bar X^N_t(\ome); t \in D\}$ is uniformly continuous in $t$. Note that $\Omega''_*:= \cup_{N=1} \Omega''_N$ satisfies $\PP''(\Omega''_*)=1$. Thus let $\Omega^*:= \cap_{N=1} \Omega^N \cap (\Omega' \times \Omega''_*)$ and we have $\PP(\Omega^*)=1.$ Moreover, for any $\ome=(\ome', \ome'') \in \Omega^*,$ there exists $N \ge 1,$ such that $\bar X(\ome)= \bar X^N(\ome)$, and $\bar X^N_t(\ome)$ is continuous in $t \in D.$

Step 2. Now we get a continuous modification for $\bar X.$ For any $\ome \in \Omega^*$, let 
\begin{equation}
	\tilde{X}_t:=
	\left\{ 
\begin{array}{l}
 	 \bar X_t,  \text{ if } t \in D, \\
 	 \lim_{ s_n \rightarrow t,   s_n \in D } X_{s_n}.
\end{array}
\right.
\end{equation} 
By the uniform continuity of $\bar X_t(\ome),\ t\in D$, we see that $\tilde X(\ome)$ is continuous. It remains to check that $\tilde X$ is a modification of $\bar X.$ Indeed, for any $t \in [0,1], $ there exists a sequence $\{s_n\}_n \subseteq D$ with $s_n \rightarrow t.$ Note that $\bar X_{s_n} = \tilde X_{s_n}  \stackrel{n}{\rightarrow} \tilde X_t $, a.s. On the other hand, for any $\vep>0,$ by the Fubini-Tonelli theorem and bounded convergence theorem, we have
$$
\PP(|\bar X_{s_n}- \bar X_{t}|> \vep)= \E'' [\E'[1_{| X^{\BY}_{s_n}-   X^{\BY}_{t}|>\vep}]_{\BY=\mathbf{W}(\ome'')} ] \rightarrow 0, \ \ n \rightarrow \infty,
$$ 
which implies $\bar X_{t}= \tilde   X_t$ a.s.
\end{proof}


\subsection{It\^{o} randomisation under causality}

We now see that a rough It{\^o} processes on $\tmmathbf{\Omega}'$
\[ {X^{\mathbf{Y}}_t}  (\omega') {= X^{\mathbf{Y}}_0}  (\omega') + \int_0^t
   A_s (\omega', \mathbf{Y}) d s + \int_0^t \Sigma_s (\omega', \mathbf{Y})
   d B_s (\omega') + \int_0^t (F_s, F'_s) (\omega', \mathbf{Y}) d
   \mathbf{Y}_s, \]
with causal coefficients, in the sense of assumption \eqref{ass:3}, becomes a classical
{     It\^{o} diffusion} upon Brownian and also more general {   It\^{o}} randomisation. We note that local martingale randomisation appeared naturally in Section \ref{sec:vmf}.

\begin{definition} A $(\mathfrak{F}_t)$-adapted continuous process $S$ is called $(\mathfrak{F}_t)$-{\em  {It\^o diffusion}} 
if it can be written as
  \[ \text{S$_t$} = S_0 + \int_0^t \beta_s d s + \int_0^t \gamma_s d W_s,
     \quad 0 \leqslant t \leqslant,T \]
with continuous and adapted $\beta, \gamma$, where $W$ is a $(\mathfrak{F}_t)$-Brownian motion. We drop reference to the filtration $(\mathfrak{F}_t)$ when no confusion is possible.
\end{definition} 

\begin{lemma}
  Let $S$ be an {It\^o} diffusion on $[0, T].$ 
Then for every $\alpha < 1 / 2$ and with probability one,
  the It{\^o} lift $\mathbf{S}^{\tmop{Ito}} = (S, \mathbb{S})$ has sample paths in the (Polish) rough paths space
  $(\mathscr{C}^{0, \alpha, 1}, \rho_{\alpha, 1})$, as defined in Section \ref{sec:not}.
%
%
\end{lemma}

\begin{proof}
  For $\beta, \gamma$ 
   uniformly bounded the statements is clear from Kolmorogov's
  criterion, noting a.s. consistency of quadratic covariation $[S^i, S^j], 1
  \leqslant i, j \leqslant d_S$ with rough path bracket $[\mathbf{S}]$. The
  general case follows from localisation. 
\end{proof}

%

\begin{assumption} \label{ass:5} Refining \eqref{ass:4}, assume 
$\mathbf{S} = \mathbf{S}^{\tmop{Ito}}$ is a $\mathfrak{F}_T''$-It{\^o} diffusion lifted {  (via It\^{o} integration)} to a rough path on $[0, T]$ on $\tmmathbf{\Omega}''$, yielding a measurable map 
\[ \mathbf{S} : (\Omega'', \mathfrak{F}_T'', \mathbb{P}'') \rightarrow
   (\mathscr{C}_T, \mathfrak{C}_T).  \]
 \end{assumption}
  

\begin{theorem}
  \label{thm:doublyIto}
  Assume \eqref{ass:1}+\eqref{ass:2}+\eqref{ass:3}+\eqref{ass:5} and let $\bar{X}$ be
  (the  continuous modification) of  
  $ X_t^{\mathbf{Y}} |_{\mathbf{Y} = \mathbf{S} (\omega)}$ 
  as guaranteed by Theorem
  \ref{thm:rcd}.  Then $\bar{X} = \bar{X} (\omega', \omega'')$ is indistinguishable from the
  ``doubly'' It{\^o} process
    \begin{eqnarray*}
    &  & \bar{X}_0 (\omega) + \int_0^t \bar{A}_s (\omega) d s + \int_0^t
    \bar{\Sigma}_s (\omega) d B_s (\omega) + \int_0^t \bar{F}_s (\omega) d S_s
    (\omega)\\
    & = & \bar{X}_0 (\omega) + \int_0^t (\bar{A}_s + \bar{F}_s \beta_s) d s +
    \int_0^t \bar{\Sigma}_s (\omega) d B_s (\omega) + \int_0^t \bar{F}_s
    (\omega) \gamma_s (\omega) d W_t (\omega)
  \end{eqnarray*}
  where, for $\Xi \in \{ A, \Sigma, F \}$ we write $\bar{\Xi}_s (\omega)
  \assign \Xi_s (\omega' , \mathbf{S} (\omega''))$. In particular, 
  \[ \tmop{Law} (\bar{X}_t | \mathfrak{F}''_T) (\omega) = \tmop{Law}
   (X^{\mathbf{Y}}_t) |_{\mathbf{Y} = \mathbf{S} (\omega )} \nobracket.
\]
\end{theorem}

\begin{proof}
We only discuss the case $ \beta = 0$. Since $\gamma_t$ is $\PP''$-a.s. continuous, 
  \[ \int_0^T \bar{F}_s (\omega) \gamma_s (\omega) \gamma^{\top}_s (\omega)
     \bar{F}_s^{\top} d s \leqslant \| \bar{F} \|_{\sup}^2 \int_0^T \gamma_s
     (\omega) \gamma^{\top}_s (\omega) d s < \infty \]
  which makes the last stochastic integral well-defined. 
By assumption \eqref{ass:5}, we have that $\bar \Xi(\ome)$ is $\mathfrak{F}_t$-adapted. 

We put $A_{i,n}^t(\ome', \mathbf{Y})= \frac{i}{2^n} \text{Leb}\left( \left\{r\in[0,t]: A_r(\ome', \mathbf{Y}) \in [ \frac{i}{2^n},  \frac{i+1}{2^n}) \right\}  \right)$. 
For any $(t,\mathbf{Y}),$ we have
$$
\int_0^t A_r(\ome', \mathbf{Y}) dr
=:\lim_{n \rightarrow \infty} \sum_{i=-n2^n}^{n2^n-1}  A_{i,n}^t(\ome', \mathbf{Y}),
$$
for each $\omega'$ and hence in $\mathbb{P}'$-probability.
Then by \cite[Corollary 8.8]{FLZ24x},  
$$
\left( \int_0^t A_s(\ome', \mathbf{Y}) ds \right) \Big|_{\mathbf{Y}=\mathbf{S}(\ome'')} =   \lim_{n \rightarrow \infty} \sum_{i=-n2^n}^{n2^n-1}   A_{i,n}^t(\ome', \mathbf{S}(\omega'')) \quad \text{in
     $(\mathbb{P}' \otimes \mathbb{P}'')$-probability. }
$$
On the other hand, the limit above is $\int_0^t \bar A_s(\ome) ds$, which implies the equivalence of the Lebesgue integral part. 
Remark that, in case that $A$ above is continuous in $t$, we could have proceeded by Riemann sum approximations, reaching the same conclusion. 
In the same spirit, using classical Riemann sum approximations to stochastic integrals for integrands with continuous sample paths, we can see that 
if $\Sigma$ is continuous for any $(\ome', \mathbf{Y})$, we have $\int_0^t \Sigma_s(\ome', \mathbf{Y}) dB_s |_{\mathbf{Y}=\mathbf{S}(\ome'')} = \int_0^t \bar \Sigma_s(\ome) dB_s .$ For the general case, since (A2.opt) implies (A2.prog), it suffices to consider the latter. 
Without loss of generality, assume $\Sigma$ is bounded (otherwise one may apply a localisation and repeat a similar argument as below). Consider 
$$
\sigma_r^{n}(\ome', \BY):= n\int_{r-\frac{1}{n}}^{r}\Sigma_{s\vee0}(\ome', \BY)ds,
\ \ \ r\in [0,T].
$$
It is evident that $r\mapsto \sigma^{n}(\ome', \BY)$ is continuous. Moreover, by the Lebesgue differentiation theorem, for each $(\omega',\BY)$,   $\lim_{n} \sigma^n_r(\ome',\BY) = \Sigma_r(\ome',\BY)$ for a.e. $r\in[0,T]$. 
Then for any $t,\BY,$ by the Lebesgue dominated convergence theorem, 
$$
\lim_n \int_0^t |\sigma^n_r - \Sigma_r|^2(\ome',\BY) dr =0
$$
for each $\omega'$ and hence in  $\mathbb{P}' $-probability. 
This implies that $\int_0^t \sigma^n_r (\ome',\BY) dB_r $ converges to $\int_0^t \Sigma_r (\ome',\BY) dB_r $ in
     $\mathbb{P}' $-probability. Note that 
 $$
 \left( \int_0^t \sigma^n_r (\ome',\BY) dB_r \right) \Big|_{\BY= \mathbf{S}(\ome'')} = \int_0^t \sigma^n_r (\ome',\mathbf{S}(\ome'')) dB_r,
 $$
and by a similar argument, $\int_0^t \sigma^n_r (\ome',\mathbf{S}(\ome'')) dB_r$ converges in $\PP'\otimes \PP''$-probability to $\int_0^t \bar \Sigma_r (\ome) dB_r$. We obtain the equivalence from the stochastic integral part again by \cite[Corollary 8.8]{FLZ24x}.

It remains to show the randomised stochastic rough integral equal to the stochastic integral driven by $W$, which follows along the lines of {\cite[Proposition 8.1]{FLZ24x}}. Indeed, by the convergence of rough stochastic Riemann sum and  \cite[Corollary 8.8]{FLZ24x}, we have 
 \begin{equation}\label{eq:r-riem}
    \left( \int_0^t (F, F')^{\mathbf{Y}} (\omega', \mathbf{Y}) d \mathbf{Y} \right)  \Big|_{\mathbf{Y} = \mathbf{S} (\omega'')} = \lim_{ |\pi| \rightarrow 0} \sum_{[u,v] \in \pi}  (\bar F_u \delta S_{u,v}+   \bar F'_u \mathbb{S}_{u,v}) (\omega', \ome'')  \quad \text{in
     $\PP$-probability. } 
\end{equation}
On the other hand, note that for any $\vep>0,$
$$
\PP \left( \left|\sum_{[u,v] \in \pi}\bar F'_u \mathbb{S}_{u,v}  \right|> \vep  \right) = \E'' \left( \PP' \left( \left|\sum_{[u,v] \in \pi}  F'_u(\ome',\mathbf{Y}) \mathbb{S}_{u,v}  \right|> \vep  \right)_{\mathbf{Y}=\mathbf{S}(\ome'')} \right),
$$
the right-hand side of which converges to null as $|\pi|$ vanishes by the bounded convergence theorem. It follows that the left-hand side of \eqref{eq:r-riem} is the limit in $\PP$-probability of $\sum_{[u,v] \in \pi}  \bar F_u \delta S_{u,v}$, which is exactly $\int_0^t \bar F_s dS_s$ by classical stochastic calculus, and hence completes our proof.
(The final ``in particular''  is immediate from Theorem \ref{thm:rcd}.) 
\end{proof}

\begin{remark} Concerning the It\^o and Lebesgue integral in the last proof. To see that, $\mathbb{P}$ a.s.,
\[ \left( \int_0^t \Sigma_r (\ome', \BY) dB_r \right) |_{\mathbf{Y} =
   \mathbf{S} (\ome'')}  = \int_0^t \bar{\Sigma}_r (\ome) dB_r \]
we can also observe that the desired equality follows by direct inspection for
``simple'' integrands of the form
\[ \Sigma_r (\ome', \BY) = \sum_i a_i (\omega') b_i \left( \BY^{t_{i - 1}}
    \right) 1_{(t_{i - 1}, t_i]}, \quad a_i {\in \mathfrak{F}_{t_{i
   - 1}}}  . \]
A monotone class argument extends this first to all bounded progessive
integrands, a truncation argument to give the equality in the full generality
of $a (\Sigma \Sigma^T)^{\mathbf{Y}} \in L^1 ([0, T]),$ $\mathbb{P}'$ a.s.
Similarly, but easier, one sees
\[ \left( \int_0^t A_s (\ome', \mathbf{Y}) ds \right) |_{\mathbf{Y} =
   \mathbf{S} (\ome'')} = \int_0^t \bar{A}_s (\ome) ds. \]
\end{remark}

\section{Applications to rough SDEs}

Let  $\mathbf{Y} \in \mathscr{C}^\alpha([0,T])$. In \cite{FHL21x} sufficient conditions on the coefficients,
\[ \Theta \in \{ b, \sigma, (f, f') \}, \]
were given to ensure  that, for
\[  d X_t = b_t (X_t, \omega ; \mathbf{Y }) d t + \sigma_t
   (X_t, \omega ; \mathbf{Y }) d B_t + (f_t, f'_t) (X_t, \omega ; \mathbf{Y}) d
   \mathbf{Y}, \qquad X_0 = \xi (\omega, \mathbf{Y}) \]
has a unique solution $X = X^{\mathbf{Y}}$, in the precise sense of Definition 4.2 in \cite{FHL21x}, which included
the requirements that 
\begin{itemizeminus}
  \item all coefficients are progressive; i.e. $(t, \omega) \mapsto
  \Theta^{\mathbf{Y}}_t (\omega)$ is progressively measurable
  
  \item $t \mapsto b_t (X_t, \omega ; \mathbf{Y }), t \mapsto (\sigma
  \sigma^T )_t (X_t, \omega ; \mathbf{Y }) \in L^1 ([0, T])$ for a.e. $\omega$
  
  \item $t \mapsto (F, F')^{\mathbf{Y}}_t : = (f_t, (D f_t) f_t + f'_t)
  (X_t, \omega ; \mathbf{Y})  $ is stochastic controlled
  w.r.t. $Y$.
\end{itemizeminus}
This says precisely {  (cf. Remark \ref{rem:A1forRSDE}, part (ii), and references therein) } that rough SDE solutions are rough It{\^o} processes that
satisfy condition \eqref{ass:1}.

Concerning the correct measurability, especially with regard to
$\mathbf{Y}$, we recall that $X ^{\mathbf{Y}}$was constructed as $n
\rightarrow \infty$ limit in probability, uniformly on $[0, T]$, of its
$[n]$-th Picard iteration, starting from
\[ X_t^{[0]} = \xi^{\mathbf{Y}} + f_0 (\xi^{\mathbf{Y}}) (Y_t - Y_0) . \]
The following two lemmas are straight-forward, we refer to \cite{FLZ24x} for
details. (The analogous statement assuming coefficients to be $\mathfrak{B}^{d_X}_T \otimes \mathfrak{C}_T$-progressively measurable also hold, cf \cite{HZ25p}.) {  As before, $\mathscr{C} = \mathscr{C}_T ( \hookrightarrow \mathscr{C}^\alpha([0,T]))$ is a Polish space of rough paths with Borel sets $\mathfrak{C}_T$, cf. Section \ref{sec:not}.} 

\begin{lemma} (Joint measurability)
 If $b, \sigma, (f, f')$ are $\mathfrak{B}^{d_X}_T \otimes \mathfrak{C}_T$-optional (cf. Assumption 4.3 in \cite{FLZ24x}), and $\xi$ is $\mathfrak{F}_0 \otimes \mathfrak{C}_T$-measurable, then the
  \[ (t, \omega, \mathbf{Y}) \mapsto b_t (X^{\mathbf{Y}}_t, \omega ;
     \mathbf{Y }), \sigma_t (X^{\mathbf{Y}}_t, \omega ; \mathbf{Y }), (f_t,
     f'_t) (X^{\mathbf{Y}}_t, \omega ; \mathbf{Y}) \]
  are $\mathfrak{C}_T$-optional so that $X = \{ X^{\mathbf{Y}}
  : \mathbf{Y} \in \mathscr{C}_T \}$ satisfies condition \eqref{ass:2}.
\end{lemma}

\begin{lemma} (Causality)
  If $b, \sigma, (f, f')$ are causal w.r.t. $\mathbf{Y}$, also $\xi$ only depends on $\mathbf{Y}$ through $\mathbf{Y}_0$, then the 
  \[ (t, \omega, \mathbf{Y}) \mapsto b_t (X^{\mathbf{Y}}_t, \omega ;
     \mathbf{Y }), \sigma_t (X^{\mathbf{Y}}_t, \omega ; \mathbf{Y }), (f_t,
     f'_t) (X^{\mathbf{Y}}_t, \omega ; \mathbf{Y}) \]
     are causal, so that $X = \{ X^{\mathbf{Y}} : \mathbf{Y} \in \mathscr{C}_T
  \}$ satisfies condition \eqref{ass:3}.
\end{lemma}

\begin{theorem} \label{thm:RSDEareRIP}RSDE solutions are rough It\^o processes, which satisfy conditions \eqref{ass:1} and, under the above conditions put forward in the previous lemmas, \eqref{ass:2} and \eqref{ass:3}, respectively.
\end{theorem} 
All statements about randomisation and conditioning of rough
It{\^o} processes etc hence apply directly to solutions of rough SDEs. In particular, 
Theorem \ref{thm:doublyIto} (see Proposition 8.2 in \cite{FLZ24x} for similar statements in the case of Brownian randomisation)
  on {     It\^{o}} randomisation of RSDEs with measurable and causal coefficients
  solutions (doubly stochastic) SDE, with notation as before, satisfies
  \[  d \bar{X}_t = b_t (\bar{X}_t, \omega ; \mathbf{S}_{.
     \wedge t} (\omega)) d t + \sigma_t (\bar{X}_t, w ; \mathbf{S}_{. \wedge
     t} (\omega)) d B_t + f_t (\bar{X}_t, w ; \mathbf{S}_{. \wedge t}
     (\omega)) d S_t.\]
  (Remark that $\mathbf{S}$, the {     It\^{o}} lift of $S$, is measurably determined by $S$ so that, from measurability perspective, there is little difference
  between the written $\mathbf{S}_{.
     \wedge t} (\omega)$ dependence and classical causal dependence in $S$.)

\section{Applications to McKean-Vlasov RSDEs}

In \cite{FHL25p}, sufficient conditions on the coefficients
\[ \Theta \in \{ b, \sigma, f \}, \]
are given to ensure well-posedness of McKean-Vlasov RSDEs, of the
form
\begin{equation}
  \label{eq.mkrsde} 
  d X^{\mathbf{Y}}_t = b (X^{\mathbf{Y}}_t, \mu^{\mathbf{Y}}_t) {d t
  + \sigma}  (X^{\mathbf{Y}}_t, \mu^{\mathbf{Y}}_t) d B_t + f 
  (X^{\mathbf{Y}}_t, \mu^{\mathbf{Y}}_t) d \mathbf{Y}_t, \quad
  \mu_t^{\mathbf{Y}} \assign \tmop{Law} (X^{\mathbf{Y}}_t)
\end{equation}
with a given initial datum in  $L_q(\Omega)$ for some $q\ge0$.
Similar to the relationship between \eqref{comMKV} and \eqref{eqn.Npart}, it is shown in \cite{BFHL25p} that \eqref{eq.mkrsde} is the mean-field ($N\to\infty$)  limit of interacting system of RSDEs
\begin{align}\label{eq.Npartrsde}
  {d X_t^{N, i;\mathbf{Y}}} = b ( {{X_t^{N, i;\mathbf{Y}}}} ,
  \mu^{N;\mathbf{Y}}_t  ) {d t + \sigma }  ( {{X_t^{N,
  i;\mathbf{Y}}}} , \mu^{N;\mathbf{Y}}_t  ) d B^i_t + f 
   ( {{X_t^{N,
  i;\mathbf{Y}}}} , \mu^{N;\mathbf{Y}}_t  ) d \mathbf{Y}_t.
\end{align} 
Here, $\mu^{N;\mathbf{Y}}_t$ is the empirical measure of $\{ {X_t^{N, i;\mathbf{Y}}} : 1
\leqslant i \leqslant N \}$.
As before, for each  $\mathbf{Y} \in \mathscr{C}_T$, the very notion of 
solutions requires that
\begin{itemizeminus}
  \item $t \mapsto b (X^{\mathbf{Y}}_t,\ \mu^{\mathbf{Y}}_t), t \mapsto (\sigma \sigma^T
  ) (X^{\mathbf{Y}}_t, \mu^{\mathbf{Y}}_t) \in L^1 ([0, T])$ for a.e.
  $\omega$
  
  \item $t \mapsto (F, F')^{\mathbf{Y}}_t : = (\hat f, (D \hat f) f )  (\mathcal X_t^{\mathbf{Y}})$ is stochastic controlled w.r.t. $Y$.
\end{itemizeminus}
This says precisely that solutions to  McKean-Vlasov RSDEs are rough It{\^o}
processes that satisfy condition \eqref{ass:1}.

Concerning the correct measurability, especially with regard to
$\mathbf{Y}$, we recall that $X ^{\mathbf{Y}}$ can be obtained as $N
\rightarrow \infty$ limit in probability, uniformly on $[0, T]$, of its
 particle approximation \eqref{eq.Npartrsde}. Consequently, as seen in \cite{BFHL25p}, the (rough It\^o) coefficients $(A^i,\Sigma^i,F^i,\hat F^{i,\prime})$ in \eqref{eq.Npartrsde}, i.e. of the  tagged particle $i$, converges to those in \eqref{eq.mkrsde}. 
Due to the results of the previous section to this system of rough SDEs, Theorem \ref{thm:RSDEareRIP} and its preceding lemmas, the
following can be seen. 

\begin{theorem} \label{thm:MKVisRIP}
  The solution $X ^{\mathbf{Y}}$  to the rough McKean-Vlasov SDEs \eqref{eq.mkrsde} is a rough It{\^o} process whose coefficient fields satisfies
  the measurability condition \eqref{ass:2} and 
  the causality condition \eqref{ass:3}.
\end{theorem} 



Let us consider the randomisation of   with  Brownian noise $\mathbf{W}
= \left( W, \int \delta W \otimes \tmop{dW} \right)$, independent of $B$. We
know from Theorem \ref{thm.randomise}  that (up to modification)
\[ \bar{X} = X^{\mathbf{Y}}  |_{\mathbf{Y} = \mathbf{W}} \nobracket \]
solves
\[ d \bar{X} = b_t (\bar{X}_t, \bar{\mu} _t) {d t + \sigma_t}  (\bar{X},
   \bar{\mu} _t) d B_t + f_t (\bar{X}, \bar{\mu} _t) d W_t, \quad \bar{\mu} _t
   := \tmop{Law} (X^{\mathbf{Y}}_t) |_{\mathbf{Y} =
   \mathbf{W}}. \nobracket \]
But at the same time, we know from Theorem \ref{thm.randomise}, in conjunction with Theorem \ref{thm:MKVisRIP} that
\[ \bar{\mu} _t
   = \tmop{Law} (X^{\mathbf{Y}}_t) |_{\mathbf{Y} =
   \mathbf{W}} \nobracket = \tmop{Law} (\bar{X} _t | \nobracket
   \mathfrak{F}_T^{\mathbf{W}}) = \tmop{Law} \left( \bar{X} _t |
   \nobracket {\mathfrak{F}^W_T}  \right) = : \mathcal{L} (\bar{X} _t |
   \nobracket W) \]
which identifies $\bar{X} = X^{\mathbf{Y}}  |_{\mathbf{Y} =
\mathbf{W}} \nobracket$ as a solution to the conditional McKean-Vlasov
equation
\[ d X_t = b (X _t, \mathcal{L} (\bar{X} _t | \nobracket W)) {d
   t + \sigma }  (X _t, \mathcal{L} (\bar{X} _t | \nobracket W)) d B_t + f  (X
   _t, \mathcal{L} (\bar{X} _t | \nobracket W)) d W_t . \]
Remark that a conditional propagation of chaos result follows along the same
logic from the propagation of chaos of the rough SDE particle system to the
rough McKean-Vlasov equation; topic of our forthcoming work \cite{BFHL25p}.

\appendix

\end{document}